\documentclass[a4paper,11pt]{article}

\usepackage[toc,page]{appendix}

\usepackage{proof}

\usepackage{hyperref}
\usepackage{url}

\usepackage{amsthm}
\usepackage{mathtools}
\usepackage{latexsym}
\usepackage{amssymb}
\usepackage{verbatim}
\usepackage{setspace}
\usepackage{qtree}
\usepackage{enumitem}
\usepackage{xcolor}
\usepackage{url}
\usepackage{stmaryrd}
\usepackage{hyphenat}

\usepackage{authblk}

\usepackage[margin=1.4in]{geometry}

\usepackage{tikz}
\usetikzlibrary{arrows,decorations.pathmorphing,backgrounds,positioning,fit}

\usepackage{hyperref}

\newcommand{\Nset}{\mathbb{N}}
\newcommand{\uml}{\mathrm{MLU}}
\newcommand{\mlu}{\mathrm{MLU}}
\newcommand{\gmlu}{\mathrm{GMLU}}
\newcommand{\tp}{\mathrm{tp}}
\newcommand{\fs}{\mathrm{FS}}
\newcommand{\fsc}{\mathrm{FSC}}
\newcommand{\propmove}{$p$-move}
\newcommand{\lormove}{$\lor$-move}
\newcommand{\landmove}{$\land$-move}

\newcommand{\di}{\blacklozenge}
\newcommand{\bo}{\blacksquare}
\newcommand{\bomove}{$\bo$-move}
\newcommand{\bbomove}{$\bo^{< d}$-move}
\newcommand{\dimove}{$\di$-move}
\newcommand{\ddimove}{$\di^{\geq d}$-move}
\renewcommand{\AA}{\mathcal{A}}
\newcommand{\BB}{\mathcal{B}}
\newcommand{\MM}{\mathfrak{M}}
\newcommand{\G}{\mathcal{G}}
\newcommand{\mods}{\mathrm{Md}}
\newcommand{\size}{size}
\newcommand{\poshard}{\#h^+}

\newcommand{\asymptotic}{\sim}
\newcommand{\equivalence}{\equiv}

\renewcommand{\phi}{\varphi}

\newbox\gnBoxA
\newdimen\gnCornerHgt
\setbox\gnBoxA=\hbox{$\ulcorner$}
\global\gnCornerHgt=\ht\gnBoxA
\newdimen\gnArgHgt
\def\Godelnum #1{%
\setbox\gnBoxA=\hbox{$#1$}%
\gnArgHgt=\ht\gnBoxA%
\ifnum     \gnArgHgt<\gnCornerHgt \gnArgHgt=0pt%
\else \advance \gnArgHgt by -\gnCornerHgt%
\fi \raise\gnArgHgt\hbox{$\ulcorner$} \box\gnBoxA %
\raise\gnArgHgt\hbox{$\urcorner$}}

\author[1]{Reijo Jaakkola}
\author[1,2]{Antti Kuusisto}
\author[1]{Miikka Vilander}
\affil{Tampere University, Finland}
\affil[2]{University of Helsinki, Finland}

\date{}

\begin{document}

\setlength\abovedisplayskip{3pt}
\setlength\belowdisplayskip{3pt}

\title{Relating description complexity to entropy}

\theoremstyle{plain}
\newtheorem{theorem}{Theorem}[section]
\newtheorem{lemma}[theorem]{Lemma}
\newtheorem{corollary}[theorem]{Corollary}
\newtheorem{proposition}[theorem]{Proposition}
\newtheorem{fact}[theorem]{Fact}
\theoremstyle{definition}
\newtheorem{definition}[theorem]{Definition}
\newtheorem{remark}[theorem]{Remark}
\newtheorem{example}[theorem]{Example}

\maketitle

\begin{abstract}
\noindent
We demonstrate some novel
links between entropy and description complexity, a notion referring to the minimal formula length for specifying given properties. Let MLU be the logic obtained by extending propositional logic with the universal modality, and let GMLU be the corresponding extension with the ability to count. In the finite, MLU is expressively complete for specifying sets of variable assignments, while GMLU is expressively complete for multisets.
We show that for MLU, the model classes with maximal Boltzmann entropy are the ones with maximal description complexity. Concerning GMLU, we show that expected Boltzmann entropy is asymptotically equivalent to expected description complexity multiplied by the number of proposition symbols considered. To contrast these results, we show that this link breaks when we move to considering first-order logic FO over vocabularies with higher-arity relations. To establish the aforementioned result, we show that almost all finite models require relatively large FO-formulas to define them. Our results relate to links between Kolmogorov complexity and entropy, demonstrating a way to conceive such results in the logic-based scenario where relational structures are classified by formulas of different sizes.
\end{abstract}

\section{Introduction}

In this article we investigate links
between description complexity and 
entropy. 
%By the description complexity of a class of models, we mean the 
%minimal length of a formula defining that class. 
%This notion of
%course depends on the logic being studied.
%In this paper we are 
%particularly interested in the description complexity of 
%completely specified model classes, i.e., equivalence 
By description complexity of a model, we mean the minimal 
length of a formula that specifies the model up to a
maximal possible extent.
%classes of models specified based on the
%logical equivalence relation.
With a strong enough logic, this amounts to
investigating the length of formulas
specifying models up to isomorphism, but this is by no means the 
only interesting scenario. 
By the description complexity of a class of models, we mean the 
minimal length of a formula defining that class. 
%This notion of
%course depends on the logic being studied.
In this paper we are 
particularly interested in the description complexity of 
completely specified model classes, i.e., equivalence classes of logics.
The main objective of the 
paper is to point out links between description complexity and entropy.
By entropy, we refer essentially to Shannon's
entropy and the earlier
notion of Boltzmann entropy from statistical mechanics.

We first consider models with unary relational vocabularies. We
study two related logics, MLU and GMLU. The logic MLU is 
the extension of propositional logic with the universal modality $\di$,
also known as global modality. The truth definition states
that $\MM,w\vDash \di\varphi$ if $\MM,u\vDash \varphi$
for some $u$ in the domain of $\MM$. Thus, in the finite, this logic is 
tuned to specify precisely which 
%sets of 
variable assignments are
present in the model considered. 
The system GMLU is the extension of MLU with the
ability to count: we have $\MM,w\vDash \di^{\geq d}\varphi$ if $\MM,u\vDash \varphi$
for at least $d$ points $u$ in the domain of $\MM$.
%We note that this 
%logic is strong enough to define any finite model with a 
%unary vocabulary up to isomorphism.
We note that when limiting to
models with a finite unary vocabulary and a fixed finite bound on domain size, GMLU is
expressively complete, being able to define all classes of
models closed under isomorphism. While MLU can fully specify which \emph{set} of
assignments is present in a model, GMLU can lift this specification to
the level of \emph{multisets}.

Let $\tau$ be a finite unary relational vocabulary,
and let $\mathrm{Mod}_n(\tau)$ denote the
class of $\tau$-models over the fixed domain $W = \{1,\dots , n\}$. 
Let $\equiv_{\mathrm{MLU}}$ and $\equiv_{\mathrm{GMLU}}$ denote the logical equivalence
relations of MLU and GMLU over $\mathrm{Mod}_n(\tau)$. We first
prove that among the classes of $\equiv_{\mathrm{MLU}}$, the class with 
the largest description complexity is the class with the 
largest Boltzmann entropy. This means that the models with the 
largest description complexity belong to the class that has 
the largest Boltzmann entropy. 
We then move on to investigating GMLU. 
Let $\langle H_B \rangle$ denote the expected Boltzmann entropy over 
the equivalence classes of $\equiv_{\mathrm{GMLU}}$, with the probability of an
individual class being its size divided by the size of $\mathrm{Mod}_n(\tau)$.
Let $\langle C \rangle$ denote the expected description 
complexity of a model chosen randomly from $\mathrm{Mod}_n(\tau)$,
%the equivalence classes of $\equiv_{\mathrm{GMLU}}$,
and   
let $|\tau|$ denote the size of the vocabulary $\tau$.  
We will prove that 
\begin{align}\label{equationoneone}
\langle H_B \rangle \asymptotic |\tau|\langle C \rangle
\end{align}
that is, $\langle H_B \rangle$ is asymptotically equivalent to $|\tau|\langle C\rangle$. 
This gives an intimate
relationship between $\langle C\rangle$ and Boltzmann entropy.
To obtain a link to Shannon entropy, we simply
note that the
Shannon entropy of the 
distribution of models based on $\equiv_{\mathrm{GMLU}}$ is
equal to $\langle H_B \rangle - \log(|\mathrm{Mod}_n(\tau)|)$.

We then move on to investigating general 
(finite) relational vocabularies.
Our main result there is that the expected description complexity of
classes of $\mathrm{FO}$ grows
asymptotically faster with domain size than the corresponding 
expected Boltzmann entropy. To establish this result, we show that almost all models require relatively large $\mathrm{FO}$-formulas to define them.

There exist well known relationships between
Kolmogorov complexity and entropy. Notably, 
for any computable distribution, the expected Kolmogorov complexity can be linked, 
within a constant, to Shannon entropy. 
See for example \cite{grunwald}, \cite{leung}, \cite{vitanyi}, \cite{teixeira} for 
discussions of the issue. The article \cite{teixeira} discusses some
generalizations and shows, e.g., that the relationship fails in the general 
case for R\'{e}nyi and Tsallis entropies. 
Links between description lengths and entropy are fundamentally interesting, 
linking syntactic issues to semantic randomness. Most notable results in 
the field concern variants of Kolmogorov complexity. 
The aim of the current article is to provide one way of
demonstrating how these
results extend beyond the realm of binary strings and
descriptions via programs.
The link given in Equation (1) elucidates 
nicely the relationship between the syntax of GMLU and models with unary 
vocabularies. The result on FO provides contrast to this and
warns against overselling the analogy between description complexities
and entropy. However, we conjecture that even for FO, a monotone
Galois connection can be demonstrated between
description complexities and related Boltzmann entropies, but this is
left for future work for lack of space.

Concerning other related work in addition to 
the links to Kolmogorov complexity, we turn attention to the
proof techniques used in the paper. One of the main tools 
%for proving the results in the paper relate to
we use is the framework of logic-related games. We note that standard Ehrenfeucht-Fra\"{i}ss\'{e} games,
and their variants such as bisimulation games, do not suffice for
the purposes of this article. Thus we utilize \emph{formula size
games} for MLU and GMLU instead. Generally, the first formula size game was defined for propositional logic by Razborov in \cite{Razborov90}. A better known version is the game of Adler and Immerman for $\mathrm{CTL}$ in \cite{AdlerI03}. The game for MLU resembles the similar
game developed in \cite{HellaV19} which was there also used to demonstrate a nonelementary succinctness gap between modal logic 
and FO. For GMLU, we develop a suitable game from scratch.
The hard part is using the games in a suitable way. In addition to games, we also use various techniques for estimating Boltzmann entropy and description complexity,
e.g., Stirling's approximation, the weak law of large numbers and counting arguments.

\section{Preliminaries}

Let $f,g:\mathbb{N} \to \mathbb{R}_{\geq 0}$. We use $f = \mathcal{O}(g)$ to denote that $f \leq C g(n)$, for some constant $C > 0$ and large enough $n$. If we want to emphasize that the implied constant $C$ depends on some parameter $p$ (which is independent of $n$), we will write $f = \mathcal{O}_p(g)$. We use $f = \Omega(g)$ to denote that $f(n) \geq Cg(n)$, for some constant $C > 0$ and large enough $n$. Finally, we use $f = \Theta(g)$ to denote that $f = \mathcal{O}(g)$ and $f = \Omega(g)$. We say that $f$ is \textbf{asymptotically} $g$, if
$\lim_{n \to \infty} f/g = 1$
and we denote this by $f \asymptotic g$. By $\log$ we mean logarithm to base two.

The following variants of classical results will be useful for our purposes.

\begin{proposition}[Stirling's approximation \cite{feller-vol-1}]\label{prop:stirling_approximation}
    $\log(n!) = n\log(n) - n\log(e) + \Theta(\log(n))$
\end{proposition}

\begin{proposition}[Weak law of large numbers \cite{loeve1955probability}]\label{proposition:law_of_large_numbers}
    Let $(X_n)_{n \in \mathbb{N}}$ be a sequence of Bernoulli random variables with success probability $p := \Pr[X_n = 1]$. Then for every $\delta > 0$ we have that
    
    \smallskip
    
    $\lim_{n\to\infty} \Pr\bigg[\bigg|p - \frac{1}{n} \sum_{i=1}^n X_n\bigg| < \delta\bigg] = 1.$
    
    \smallskip

\end{proposition}

%Next we review the logics
%that we will study in this work. 
We next define the logics studied in this work.
Let $\tau = \{p_1, \dots, p_k\}$ be a set of proposition symbols. The syntax of %\textbf{universal modal logic with counting}
\textbf{graded universal modal logic} $\gmlu[\tau]$ is generated as follows.

\smallskip

$\varphi := \di^{\geq d} \psi \mid \bo^{< d} \psi \mid \varphi \lor \varphi \mid \varphi \land \varphi \mid \di^{\geq d} \varphi \mid \bo^{< d} \varphi$

$\psi := p \mid \neg p \mid \psi \lor \psi \mid \psi \land \psi$

\smallskip

\noindent
Here $p \in \tau$ and $d \in \Nset$. Notice that by design, the formulas of $\gmlu[\tau]$ only contain proposition symbols that occur in the scope of a global modal operator $\di^{\geq d}$ or $\bo^{< d}$. Additionally, all formulas are in negation normal form. (In the sequel, the notation $\neg \varphi$ will always mean the negation normal form formula, where the negation has been pushed to the level of literals.)
Now, let $\MM$ be a Kripke model with universe $W$.  
The semantics of the global graded modalities are defined as follows:
$(\MM, w) \vDash \di^{\geq d}\varphi \Leftrightarrow$ there are at least $d$ points $v \in W$ such that $(\MM, v) \vDash \varphi.$
Additionally, $(\MM, w) \vDash \bo^{< d} \varphi \Leftrightarrow (\MM, w) \vDash \neg \di^{\geq d} \neg \varphi$. Intuitively this means that all points in $\MM$ satisfy $\varphi$, except for less than $d$ exceptions. The rest of the semantics is defined as usual in propositional logic. Note
that $\di^{\geq d}$ and $\bo^{< d}$ are
dual to each other. (We note that in this article, modal logics will always have a strictly unary vocabulary, so Kripke models will not have an accessibility relation as part of the relational structure involved.)

Given a Kripke model $\MM$ over $\tau$ and $\varphi \in \gmlu[\tau]$, we define the point-free truth relation such that
$\MM \vDash \varphi \Leftrightarrow \text{for every } w\in W, \text{ we have } \MM,w \vDash \varphi$. Since no propositional symbol occurs outside the scope of a global modality, $\MM \vDash \varphi$ iff there is some $w\in W$ for which $\MM,w \vDash \varphi$.
Hence the truth of any formula of $\gmlu$ is independent of the evaluation point $w$. The property that truth is always independent of the evaluation point is the reason we defined $\gmlu$ so that proposition symbols must occur in the scope of modalities. 
The fragment of $\gmlu[\tau]$ where $d=1$ for all modalities is called \textbf{universal modal logic} $\mlu[\tau]$. This logic has only the modalities $\di^{\geq 1}$ and $\bo^{< 1}$, and we denote these with $\di$ and $\bo$ for simplicity.

A $1$-\textbf{type} $\pi$ over $\tau$ is a maximally consistent set of literals (propositional symbols and their negations). This means that $\pi$ has exactly one of $p$ or $\neg p$ for each $p\in\tau$.
The set of all $1$-types over $\tau$ is denoted by $\boldsymbol{\alpha}_\tau$. Given a Kripke model $\MM$ over $\tau$ and $w\in W$, we let $\tp_\MM[w]$ denote the unique $1$-type that $w$ \textbf{realizes}.

The \textbf{size} of a formula $\varphi \in \gmlu[\tau]$, denoted $\size(\varphi)$, is defined as follows:
\begin{itemize}
    \item $\size(\alpha) = 1$ for a literal $\alpha$,
    \item $\size(\varphi \lor \psi) = \size(\varphi \land \psi) = \size(\varphi) + \size(\psi) + 1$,
    \item $\size(\di^{\geq d} \varphi) = \size(\bo^{< d} \varphi) = \size(\varphi) + d$.
\end{itemize}
We emphasize that according to our definition \emph{all literals have the same size}. The motivation for this is to consider negative (i.e., negated)
information and positive (i.e., non-negated)
information as equal in
relation to formula size.
This also explains the convention of defining $\gmlu$ such that formulas are in negation normal form. 
%\textcolor{red}{I modified the above..., we can discuss this.}
%The motivation for this convention is to guarantee
%that all the $1$-types 
%(over some fixed set of %proposition symbols) have the same size, when %viewed as conjunctions of literals.

We will also consider standard first-order logic $\mathrm{FO}$. Let $\tau = \{R_1,\dots,R_k\}$ be a set of relation symbols. The syntax of $\mathrm{FO}[\tau]$ is generated by the following grammar:

\smallskip

%\[
$\varphi := x = y \mid \neg x = y \mid R(\overline{x}) \mid \neg R(\overline{x}) \mid \varphi \lor \varphi \mid \varphi \land \varphi \mid \exists x \varphi \mid \forall x \varphi$,

\smallskip

%\]
\noindent
where $\overline{x}$ is a tuple of variables.
We use the standard semantics of $\mathrm{FO}[\tau]$. 
The \textbf{size} of a formula $\varphi \in \mathrm{FO}[\tau]$, denoted $\size(\varphi)$, is defined as follows:
\begin{itemize}
    \item $\size(\alpha) = 1$ for a literal $\alpha$,
    \item $\size(\varphi \lor \psi) = \size(\varphi \land \psi) = \size(\varphi) + \size(\psi) + 1$,
    \item $\size(\exists x \varphi) = \size(\forall x \varphi) = \size(\varphi) + 1$
\end{itemize}
Again we emphasize that according to our definition, all literals have the same size.

%It is easy to see that $\gmlu$ is closely related to \emph{monadic} $\mathrm{FO}$, i.e., $\mathrm{FO}$ over unary vocabularies: there is a polynomial time translation from $\gmlu$ to monadic $\mathrm{FO}$ and the two logics are equi-expressive. However, there are also important differences. For instance, according to our definitions, $\gmlu$ can count more succinctly. Indeed, the formula $\di^{\geq n} \top \land \neg \di^{\geq n + 1} \top$, which has size roughly $2n$, expresses that the size of the model is $n$, while expressing the same property in monadic $\mathrm{FO}$ seems to require a formula of size at least $n^2$.

Let $\mathcal{L} = (L,\vDash)$ be a logic and $\mathcal{M}$ a \emph{finite class of models}. The class $\mathcal{M}$ is here considered fixed and known from the context. We say that a formula $\varphi \in L$ of $\mathcal{L}$ \textbf{defines} a set $M \subseteq \mathcal{M}$ if for all $\mathfrak{M} \in \mathcal{M}$, we have $\mathfrak{M} \vDash \varphi$ iff $\mathfrak{M} \in M$. Such a set $M$ is called $\mathcal{L}$-\textbf{definable} (with respect to $\mathcal{M}$). Given an $\mathcal{L}$-definable set $M$, its $\mathcal{L}$-\textbf{description complexity} $C_\mathcal{L}(M)$ is the size of a minimum size formula $\varphi \in \mathcal{L}$ which defines $M$. Now, if $\mathcal{L}$ is closed under negation (as all the logics in this paper are), then the relation ``$\mathfrak{M}$ and $\mathfrak{N}$ satisfy the same $\mathcal{L}$-formulas''
induces a partition of $\mathcal{M}$ denoted by $\equivalence_\mathcal{L}$. The $\mathcal{L}$-description complexity \emph{of a model} $\MM$ with respect to $\equivalence_{\mathcal{L}}$ is $C_{\equivalence_{\mathcal{L}}}(\MM) := C_{\mathcal{L}}(M)$, where $M$ is the equivalence class of $\MM$. For brevity, we formulate the results below only for description complexities of classes rather than models. 
%is the description complexity of the class $M$ of the model $\MM$ in the partition $\equivalence_{\mathcal{L}}$.

Let $\mathcal{M}$ be a finite class of models and let $\equivalence$ be an \emph{arbitrary} equivalence relation over $\mathcal{M}$. Given an equivalence class $M \subseteq \mathcal{M}$, we define its \textbf{Boltzmann entropy} as $H_B(M) := \log(|M|)$.
%
%
%
%$H_B(M) := \log(|M|)$.
%
%
%
This terminology is borrowed from statistical mechanics, where the Boltzmann entropy of a macrostate is the quantity $k_B \ln(\Omega)$. Here $k_B$ is the Boltzmann constant, $\ln$ the natural logarithm and $\Omega$ the number of microstates associated with the macrostate. Note that in our definition, we use the binary logarithm. 
As a measure of randomness, it is natural to 
define the Boltzmann entropy \emph{of a model} $\MM$ as $H_B^{\equivalence}(\MM) := H_B(M)$, where $M$ is the equivalence class of $\MM$. This reflects the \emph{informal intuition} that often the randomness of an object $x$ is in fact more related to the randomness of a similarity class of objects that $x$ belongs to rather than to $x$ itself. Consider, for example, the equivalence classes that a
sufficiently weak logic defines over the
universe of binary strings of a fixed finite length.
In a suitable logic, the string with only bits 1 will be in its own singleton equivalence class. So will the string with only bits 0. Also the two strings with the strictly alternating pattern …010101… are likely to be in their own singleton classes. But more ``random’’ strings end up in larger classes, and each such class is a similarity class for its member strings. It is natural to consider the strings in the same similarity class as equally random. Pushing this perspective, they could perhaps even be considered---in some informal sense---the ``same'' random string with the same degree of randomness. This degree can be measured by the size of the class, or by the binary logarithm of the size of the class. Thus it is natural to define the Boltzmann entropy of a single string as the Boltzmann entropy of the similarity class it belongs to.

%\textcolor{red}{This paragraph should be more clear.}
Let $\{M_i \mid i \in I\}$ enumerate the equivalence classes of $\equivalence$. As they form a partition of $\mathcal{M}$, we have the following natural probability distribution over the equivalence classes:
$p_\equivalence(M_i) := |M_i|/|\mathcal{M}|$. Given a random variable $X:\{M_i \mid i \in I\} \to \mathbb{R}_{\geq 0}$, we use $\langle X \rangle$ to denote its expected value with respect to $p_\equivalence$. 
Now, suppose we are in a context where we have fixed a 
finite universe $\mathcal{M}$ of models. Let $\equiv_{\gmlu}\, \subseteq\, \mathcal{M}\times \mathcal{M}$ be the corresponding equivalence relation of $\gmlu$.
Suppose $\{M_i \mid i \in I\}$
enumerates the equivalence classes of $\equiv_{\gmlu}$. Recall that $C_{\gmlu}(M_i)$
denotes the $\gmlu$ description complexity of the class $M_i$.
Let $p_{\equivalence_{\gmlu}}(M_i)$ be the corresponding 
probability $|M_i|/|\mathcal{M}|$.
In this paper, we denote by $\langle C \rangle$ the
expected description complexity of $\gmlu$, that is,
$\langle C\rangle = 
\sum_{i \in I} p_{\equivalence_{\gmlu}}(M_i) C_{\gmlu}(M_i)$.
The class $\mathcal{M}$ will be clear from the context. Note
that trivially the same expected value is obtained for the description
complexity of \emph{models} over $\mathcal{M}$ if we
give every model $\mathfrak{M}\in \mathcal{M}$
the probability $1/|\mathcal{M}|$ (the uniform distribution).

%The same expected value $\langle C \rangle$ is 
%obtained for $C_{\equivalence_{\gmlu}}$ over the
%uniform distribution %on $\mathrm{Mod}_n(\tau)$.
%

The expected value $\langle H_B \rangle$ of $H_B$ with respect to the distribution $p_\equivalence$ is closely related to the \textbf{Shannon entropy} $H_S(\equivalence)$ of $\equivalence$, which we define as the expected value of the random variable $M_i \mapsto -\log(p_\equivalence(M_i))$. More explicitly, we define that
$H_S(\equivalence) := - \sum_{i \in I} p_\equivalence(M_i) \log(p_\equivalence(M_i))$.
Note that this expression is always well-defined, since $M_i \neq \varnothing$, for every $i \in I$. The following result is established in Appendix \ref{appendix:shannon_boltzmann}. Note that the expected value of $H^\equivalence_B$ over the uniform distribution on $\mathcal{M}$ is equal to $\langle H_B \rangle$, so the result could also be formulated for single models.

\begin{proposition}\label{prop:shannon_boltzmann}
   Let $\mathcal{M}$ be a finite class of models and $\equivalence\, \subseteq \mathcal{M}\times \mathcal{M}$ an
equivalence relation over $\mathcal{M}$. Then
$H_S(\equivalence) + \langle H_B \rangle = \log(|\mathcal{M}|)$.

\end{proposition}

\noindent
There exist 
results in the literature on entropy similar to the above, see, e.g.,
\cite{conceptsinthermalphysics} and \cite{zupanovic}. 
By the proposition, both the
Shannon entropy of $\equivalence$ and the expected Boltzmann entropy of $\equivalence$ cannot be simultaneously 
large (meaning close to their maximum value $\log(|\mathcal{M}|)$).
Indeed, suppose we do not alter $\mathcal{M}$, so $\log(|\mathcal{M}|)$ is constant. Now suppose we 
alter $\equivalence$ so that $H_S(\equivalence)$ is increased. This lowers $\langle H_B \rangle$. Vice versa, increasing $\langle H_B \rangle$ lowers $H_S(\equivalence)$. Shannon entropy and expected Bolzmann entropy are  complementary quantities, summing to a constant.

\section{MLU: The largest class has maximal description\\ complexity}

Fix $\tau = \{p_1,\dots,p_k\}$ and let $\mathrm{Mod}_n(\tau)$ be the set of Kripke models over $\tau$ and the \emph{fixed universe} $W~=~\{1, \dots, n\}$. 
In this section we consider the equivalence $\equivalence_{\uml[\tau]}$ as defined above and denote it by $\equivalence$. We show that in this canonical partition, the largest class, which is the one with the largest Boltzmann entropy, has maximal $\uml[\tau]$-description complexity.
%Given two Kripke models $\MM_1$ and $\MM_2$ of the same size, we have by definition that
%
%\[\MM_1 \sim_{\mathrm{MLU}[\tau]} \MM_2 \Leftrightarrow \text{for every } \varphi \in \mlu[\tau] \text{ we have that } (\MM_1 \vDash \varphi \Leftrightarrow \MM_2 \vDash \varphi).\]
%
%In this section we will use $\sim$ to denote $\sim_{\mathrm{MLU}[\tau]}$. 

The equivalence classes of $\equivalence$ can be described easily. For Kripke models $\MM_1, \MM_2 \in \mathrm{Mod}_n(\tau)$, we have
$\MM_1 \equivalence \MM_2 \Leftrightarrow \{\tp_{\MM_1}[w] \mid w \in W\} = \{\tp_{\MM_2}[w] \mid w\in W\}$.
That is, each equivalence class is uniquely determined by the $1$-types realized in it. As the number of $1$-types over $\tau$ is $2^k$, the number of equivalence classes of $\equivalence$ is $2^{2^k} - 1$. Given a set $\Pi \subseteq \boldsymbol{\alpha}_\tau$, we let $M_\Pi$ be the equivalence class that has the models that realize exactly the $1$-types in $\Pi$.

Note the equivalence class $M_\Pi$ of any set $\Pi \subseteq \boldsymbol{\alpha}_\tau$ can be defined by the following formula:

\smallskip 

$
\varphi(\Pi) := \bigwedge\limits_{\pi \in \Pi} \di \psi(\pi) \land \bo (\bigwedge\limits_{\pi \in \boldsymbol{\alpha}_\tau \setminus \Pi} \neg \psi(\pi)),
$

\smallskip

\noindent
where $\psi(\pi)$ is the conjunction of the literals in the 1-type $\pi$. For $\Pi \neq \boldsymbol{\alpha}_\tau$, the size of the formula $\varphi(\Pi)$ is $k2^{k+1}+|\Pi|$. For $\Pi = \boldsymbol{\alpha}_\tau$, the size is $k2^{k+1} + 2^k - 1$. We see that the classes with at most one type missing are tied for the largest formula size.

%Assuming $I \neq \boldsymbol{\alpha}_\tau$, the formula $\varphi(I)$ has $k2^k$ literals, $(k-1)2^k$ conjunctions between literals and $2^k-1$ conjunctions and disjunctions elsewhere, and $|I|+1$ modal operators. In total, the size is $k2^{k+1}+|I|$. If $I = \boldsymbol{\alpha}_\tau$, then the size is $k2^{k+1} + 2^k - 1$. We see that the classes with at most one type missing are tied for the largest formula size.

Using, e.g., standard probabilistic arguments, one can show that for Kripke models of size $n$, where $n$ is much larger than $2^k$, the largest equivalence class is the one realizing all the $1$-types. In fact, the largest class will contain ``almost all'' of the models of size $n$.

\begin{proposition}\label{prop:largest_equivalence_class}
    If $n$ is large with respect to $k$, then $|M_\Pi| < |M_{\boldsymbol{\alpha}_\tau}|$,
    for every $\Pi \subset \boldsymbol{\alpha}_\tau$.
\end{proposition}

On the other hand, we can show that the equivalence class containing models which realize all the $1$-types is one of the most difficult ones to define. To prove this, we will start by introducing a formula size game for $\uml[\tau]$.

The formula size game for $\uml[\tau]$, denoted $\fs^\tau_{r_0}(\AA_0, \BB_0)$ has two players: Samson and Delilah. We refer to them as S and D, or he and she, respectively. The game has three parameters: a natural number $r_0 \geq 1$ and two sets of Kripke-models $\AA_0$ and $\BB_0$. Positions of the game are of the form $(r, \AA, \BB)$ and the starting position is $(r_0, \AA_0, \BB_0)$. 

In each position, S makes a move. The moves available for S in position $(r, \AA, \BB)$ are: 
%
%The moves available for him to make in a position $(r, \AA, \BB)$ are %the following:
\begin{itemize}
    \item \propmove: S chooses a $\tau$-literal $\alpha$. The game ends. If $\AA \vDash \alpha$ and $\BB \vDash \neg \alpha$, then S wins. Otherwise D wins. S cannot make this move if he has not made a \dimove\ so far.
    \item \lormove: S chooses $\AA_1, \AA_2 \subseteq \AA$ such that $\AA_1 \cup \AA_2 = \AA$ and $r_1, r_2 \geq 1$ such that $r_1 + r_2 + 1 = r$. D chooses whether the next position is $(r_1, \AA_1, \BB)$ or $(r_2, \AA_2, \BB)$.
    \item \landmove: The same as a \lormove\ with the roles of $\AA$ and $\BB$ switched.
    %\item \negmove: The next position of the game is $(r, \BB, \AA)$. S cannot make this move twice in a row.
    \item \dimove: For every $(\MM, w) \in \AA$, S chooses $v \in W$. Let $\AA'$ be the set of models $(\MM, v)$ chosen this way. Let 
    $\BB' := \{(\MM, v) \mid (\MM, w) \in \BB \text{ for some } w \in W, v \in W\}$.
    The next position of the game is $(r-1, \AA', \BB')$. S cannot make this move if $r = 1$.
    \item \bomove: The same as a \dimove\ with the roles of $\AA$ and $\BB$ switched.
\end{itemize}

\begin{theorem}\label{pelitoimii}
   The following statements are equivalent:
   \begin{enumerate}
       \item S has a winning strategy in the game $\fs^\tau_r(\AA, \BB)$.
       \item There is $\varphi \in \uml[\tau]$ with size at most $r$ such that $\AA \vDash \varphi$ and $\BB \vDash \neg\varphi$.
   \end{enumerate}
\end{theorem}
\begin{proof}
    Simple proof by induction. A version for basic modal logic can be found in \cite{HellaV19}.
\end{proof}

Suppose that $\pi_1,\dots,\pi_n$, where $n = 2^{|\tau|}$, enumerates all the $1$-types over $\tau$. Let $\MM_0$ denote a Kripke model with domain $\{1,\dots,n\}$ and with the property that for every $1\leq i\leq n$ the $1$-type realized by $i$ is $\pi_i$. For every $i\neq j$, we let $\MM_{i,j}$ denote the Kripke model obtained from $\MM_0$ by specifying that the $1$-type of $i$ is $\pi_j$.  We further denote $\MM_i := \MM_{i,1}$ for $2 \leq i \leq n$ and $\MM_1 := \MM_{1,2}$. Each model $\MM_i$ is now missing the type $\pi_i$ and is otherwise identical to $\MM_0$. We let $\AA_0 = \{(\MM_0,1)\} \text{ and } \BB_0 = \{(\MM_{i},1) \mid 1 \leq i \leq n\}$.
%\[\mathcal{A} = \{(\MM,i) \mid 1\leq i\leq n\}\]
%
%and
%
%\[\mathcal{B} = \{(\MM_{i,j},i) \mid i\neq j\}.\]
%
We will next show that separating these two sets requires a large $\uml[\tau]$ formula. 

%During the game S can make \negmove s to switch the sides of the sets of models. As none of the operators we use modify the models, one of the sets will always only have versions of the model $\MM_0$. We denote this set of models by $\AA$. We denote the set of models on the opposite side by $\BB$.

\begin{lemma}\label{lem:umldwins}
    D has a winning strategy in the game $\fs^\tau_{k2^{k+1}+2^k-2}(\AA_0,\BB_0)$.
\end{lemma}
\begin{proof}
    We use the following notation for the set of different underlying models that occur in a set $X$ of pointed models:
    $\mods(X) = \{\MM \mid (\MM, i) \in \AA \text{ for some } i\}$.
    
    We define a measure for a position of the game called hardness. Let $\pi_i$ be a type and let $P = (r, \AA, \BB)$ %or $P = (r, \BB, \AA)$ 
    be a position of the game. We define four different kinds of types and the hardness of those types as follows:
    \begin{enumerate}
        \item If no \dimove s have been made in the game so far, $\AA \neq \emptyset$ and $\MM_i \in \mods(\BB)$, then $\pi_i$ is of kind 1 and $h_i(P) = 2k$. 
        \item Otherwise, if there are propositionally equivalent $(\MM_0, j) \in \AA$ and $(\MM_i, l) \in \BB$, then $\pi_i$ is of kind 2 and $h_i(P) = 2k$.
        \item Otherwise, if $(\MM_0, i) \in \AA$ and $\MM_i \in \mods(\BB)$, then $\pi_i$ is of kind 3 and
        \[
        h_i(P) = 2\cdot|\{(\MM_i, j) \in \BB \mid j \neq i, \pi_i \text{ and } \pi_j \text{ differ by exactly one proposition}\}|-1.
        \]
        \item Otherwise, $\pi_i$ is of kind 4 and $h_i(P) = 0$.
    \end{enumerate}
    We further denote the number of types with positive hardness by $\poshard(P)$ and define the hardness $h(P)$ of the position $P$ as
    %\[
    $h(P) = \sum\limits_{1 \leq i \leq n}h_i(P) + \poshard(P) -1.$
    %\]
    
    \medskip
    
    We will describe the winning strategy for D in terms of maintaining the following two conditions in each position $P$ of the game:
%%%%%Check this part 
    \begin{enumerate}
        \item[(a)] $r < h(P)$,
        \item[(b)] there is at most one type of kind 3 in position $P$.
    \end{enumerate} We will show that while these conditions hold, S cannot win. Since the resource $r$ of S will run out eventually, this is a winning strategy for D.
    
    In the starting position $P_0$ no \dimove s have been made and $\MM_i \in \mods(\BB_0)$ for each $1 \leq i \leq n$ so all types $\pi_i$ are of kind 1 and have $h_i(P_0) = 2k$. Thus condition (b) holds and
    
    \medskip
    
    $
    r = k2^{k+1}+2^k-2 < k2^{k+1}+2^k-1 = 2k2^k + 2^k - 1 = h(P_0)
    $
    
    \medskip
    
    %and condition (b) holds. 

    \textbf{\propmove}: In each position $P$ of the game, we have $r \geq 1$ so while $r < h(P)$ holds, we have $h(P) \geq 2$. Using this, we show that any \propmove\  made by S while $r < h(P)$ leads to a win for D. If no \dimove s have been made, then S cannot make a \propmove. 
    If there is a type $\pi_i$ of kind 2, then there are propositionally equivalent $(\MM_0, j) \in \AA$ and $(\MM_i, l) \in \BB$ so no literal separates them. 
    If neither of the above hold, then by condition (b), there is a type $\pi_i$ of kind 3 with $(\MM_0, i) \in \AA$ and $(\MM_i, j), (\MM_i, l) \in \BB$, where $\pi_j$ and $\pi_l$ differ from $\pi_i$ by exactly one proposition. Again no literal separates $\AA$ and $\BB$. 
    
    %
    
    %\textbf{\negmove s}: Recall that we use $\AA$ and $\BB$ for the descendant sets of $\AA_0$ and $\BB_0$ regardless of which side of the game they are on. Clearly then \negmove s no dot affect hardness so the conditions are maintained through \negmove s.
    
    \textbf{\lormove}: Similar to the \landmove\ case below. Full details in the Appendix.

    \textbf{\landmove}: We show that one of the positions $P_1$, $P_2$ satisfies the conditions (a) and (b). %We only handle the case, where $\BB$ is on the right.

   Let $\BB_1, \BB_2 \subseteq \BB$ and $r_1, r_2 \geq 1$ be the choices of S. Let $\pi_i$ be a type. If $\pi_i$ is of kind 1, then  $\MM_i \in \mods(\BB_1)$ or $\MM_i \in \mods(\BB_2)$ so $\pi_i$ is still of kind 1 and $h_i(P_1) = 2k$ or $h_i(P_2) = 2k$. If $\pi_i$ is of kind 2 with propositionally equivalent models $(\MM_0, j) \in \AA$ and $(\MM_i, l) \in \BB$, then $(\MM_i, l) \in \BB_1$ or $(\MM_i, l) \in \BB_2$ so $\pi_i$ is still of kind 2 and $h_i(P_1) = 2k$ or $h_i(P_2) = 2k$.

    Finally if $\pi_i$ is a type of kind 3, then S can split the models $(\MM_i, j) \in \BB$, where $\pi_i$ and $\pi_j$ differ by one proposition, between the sets $\BB_1$ and $\BB_2$.

    Assume that S puts all these models on the same side. Then $h_i(P_1) = h_i(P)$ or $h_i(P_2) = h_i(P)$. Thus $h_i(P_1) + h_i(P_2) \geq h_i(P)$. Additionally $\poshard(P_1) + \poshard(P_2) \geq \poshard(P)$ so 
    %we obtain $h(P_1) + h(P_2) \geq h(P)-1$. Now $r_1 + r_2 = r - 1 < h(P) - 1 \leq h(P_1) + h(P_2)$ so we have $r_1 < h(P_1)$ or $r_2 < h(P_2)$.
    %
    \begin{align*}
    h(P_1) + h(P_2) &= \sum\limits_{1\leq i\leq n}h_i(P_1) +  \sum\limits_{1\leq i\leq n}h_i(P_2) + \poshard(P_1) + \poshard(P_2) - 2 \\
    &\geq \sum\limits_{1\leq i\leq n}h_i(P) + \poshard(P) - 1 - 1 = h(P) - 1.
    \end{align*}
    Now $r_1 + r_2 = r - 1 < h(P) - 1 \leq h(P_1) + h(P_2)$ so we have $r_1 < h(P_1)$ or $r_2 < h(P_2)$.
    
    Now assume that S splits some models $(\MM_i, j)$ to both sides. Now $h_i(P_1) + h_i(P_2) \geq h_i(P)-1$. In addition, the type $\pi_i$ has positive hardness in both positions $P_1$ and $P_2$ so $\poshard(P_1) + \poshard(P_2) \geq \poshard(P) + 1$. These two deviations from the above case, that only concern the single type $\pi_i$ of kind 3, cancel each other out so again $h(P_1) + h(P_2) \geq h(P) - 1$ and therefore $r_1 < h(P_1)$ or $r_2 < h(P_2)$.
    %\begin{align*}
    %h(P_1) + h(P_2) &= \sum\limits_{1\leq i\leq n}h_i(P_1) +  \sum\limits_{1\leq i\leq n}h_i(P_2) + \poshard(P_1) + \poshard(P_2) - 2 \\
    %&\geq \sum\limits_{1\leq i\leq n}h_i(P) - 1 + \poshard(P) + 1 - 1 - 1 = h(P) - 1 
    %\end{align*}
    %As above, $r_1 < h(P_1)$ or $r_2 < h(P_2)$.
    Finally, all types are of the same kind as in position $P$ so condition (b) holds.
    
    \textbf{\dimove}: %For \dimove s from a position $P$ to a following position $P'$, we show that $h(P) \leq h(P')$. Since $r' = r-1$, this implies that $r < h(P)$ is maintained. 
    %Assume S makes a \dimove\ with $\AA$ on the left and 
    Let $(\MM_0, i)$ be a choice of S. For each $j \neq i$ with $\MM_j \in \mods(\BB)$, we have $(\MM_j, i) \in \BB'$ so $\pi_j$ is a type of kind 2 and $H_j(P') = 2k$. If there are multiple versions of $\MM_0$ in $\AA$ and S makes another choice $(\MM_0,l)$, then all types $\pi_j$ with $\MM_j \in \mods(\BB)$ work the same way. If S only chooses $(\MM_0, i)$ and we have $\MM_i \in \mods(\BB)$, then $\pi_i$ becomes a type of kind 3 with $(\MM_i, j) \in \BB'$ for all $j \neq i$ so $h_i(P') = 2k-1$. We additionally note that by the definition of hardness, $h(P) \leq 2k \cdot |\mods(\BB)| + |\mods(\BB)|-1$. Thus 

\smallskip

$
    h(P') \geq 2k \cdot |\mods(\BB)|-1 + |\mods(\BB)|-1 \geq h(P)-1 > r-1 = r'
$

\smallskip

\noindent
and condition (b) is maintained.

    \textbf{\bomove}: %We again show $h(P) \leq h(P')$. 
    For each $\MM_i \in \mods(\BB)$, S chooses at least one $(\MM_i, l) \in \BB'$. Let $\pi_j$ be the type this model realizes. Now $(\MM_0, j) \in \AA'$ realizes the same type, $\pi_i$ is of kind 2 and $h_i(P') = 2k$. Thus $h(P') = 2k \cdot |\mods(\BB)|+ |\mods(\BB)|-1 \geq h(P) > r-1 = r'$ and condition (b) holds.
\end{proof}

We have shown that the largest class $M_{\boldsymbol\alpha_\tau}$ requires a formula of size at least $k2^{k+1}+2^k-1$ to define. Since any of the classes can be defined via a formula of \emph{precisely} this size, we see that in the case of $\uml[\tau]$ the largest class is maximally difficult to define.

\begin{proposition}
    %The equivalence class which has the highest $\mlu$-description complexity is the largest one, namely $M_{\boldsymbol\alpha_\tau}$.
    The largest equivalence class $M_{\boldsymbol\alpha_\tau}$ of\ \ $\equivalence_{\uml[\tau]}$ has maximal $\mlu[\tau]$-description complexity.
\end{proposition}

\begin{comment}
\begin{remark}

We briefly sketch the impact of negation normal form and formula size with no negations on the above result. 

The notion of formula size we use does not count negations. If we had free use of negation and this same notion of formula size, then the game would have a move for negation and both the lower and upper bounds would be unaffected. The following discussion concerns an alternative definition of formula size, where negations are counted.

We do not have a lower bound proof that counts negations. Thus counting negations would introduce a $\mathcal{O}(2^k)$ difference between the lower and upper bounds. We do, however, conjecture that our upper bound formulas are optimal even when counting negations so we discuss them in more detail. 

Using De Morgan laws, the formula for $M_{\boldsymbol\alpha_\tau}$ can be brought down to $2^k-1$ negations. On the other hand, for the class $M_\Pi$, where the only type missing is the one with only positive literals, the number of negations is $2^k$. It would then seem that $M_\Pi$ is more difficult to define than $M_{\boldsymbol\alpha_\tau}$ by a difference of one negation.

If we additionally remove $\bo$ as a primitive operator, then the class $M_\Pi$ above has $2^k+1$ negations, all other classes with one type missing have $2^k$ and $M_{\boldsymbol\alpha_\tau}$ still has $2^k-1$. In this case, all classes with one type missing seem to be slightly harder to define than $M_{\boldsymbol\alpha_\tau}$.

\end{remark}
\end{comment}

\section{GMLU: Relating entropy and description complexity asymptotically}

Fix $\tau = \{p_1,\dots,p_k\}$ and let $\ell = 2^k$. A Kripke model $\MM$ with universe $W = \{1,\dots,n\}$ can be described in $\gmlu[\tau]$ up to isomorphism. Hence the equivalence classes of $\equivalence_{\gmlu[\tau]}$, hereafter denoted $\equivalence$, over $\mathrm{Mod}_n(\tau)$ are the isomorphism classes. Since $\MM$ can be described up to isomorphism by listing how many times each $1$-type is realized, there is a one-to-one correspondence between isomorphism classes and tuples $(n_1,\dots,n_\ell)$, where $n_1 + \dots + n_\ell = n$. We will use $[n_1,\dots,n_\ell]$ to denote the isomorphism class consisting of those Kripke models of size $n$ in which the $i$th type is realized precisely $n_i$-times. Note that 
$|[n_1,\dots,n_\ell]| = \binom{n}{n_1,\dots,n_\ell}$.

In this section we show that the expected Boltzmann entropy $\langle H_B \rangle$ is asymptotically $|\tau|$ times the expected $\gmlu[\tau]$-description complexity with respect to the distribution $p_\equivalence$.

\subsection{Expected Boltzmann entropy}\label{entropy}

In this subsection we will establish that
$\langle H_B \rangle \asymptotic  |\tau|n$.
%The following formula, which follows from Proposition \ref{prop:stirling_approximation}, will be convenient for our purposes
%\[\log\binom{n}{n_1,\dots,n_\ell} = n\bigg(\log(n) - \sum_{i = 1}^\ell \frac{n_i}{n}\log(n_i)\bigg) + \Theta\big(\log(n)\big)\]
Using Proposition \ref{prop:stirling_approximation} we get the following alternative asymptotic formula for $\langle H_B \rangle$
\begin{align*}
        &\sum_{n_1 + \dots + n_\ell = n} p_\equivalence([n_1,\dots,n_\ell]) \log\binom{n}{n_1,\dots,n_\ell} \\
        = & \sum_{n_1 + \dots + n_\ell = n} p_\equivalence([n_1,\dots,n_\ell]) \bigg(n\bigg(\log(n) - \sum_{i = 1}^\ell \frac{n_i}{n}\log(n_i)\bigg) + \Theta\big(\log(n)\big)\bigg) \\
        = & \bigg(\sum_{n_1 + \dots + n_\ell = n} p_\equivalence([n_1,\dots,n_\ell]) \bigg(\sum_{i=1}^\ell \frac{n_i}{n} \log \bigg(\frac{n}{n_i}\bigg)\bigg)\bigg)n + \Theta\big(\log(n)\big)
\end{align*}
We will show that
\begin{equation}\label{eq:rewritten_expected_boltzmann}
    \bigg(\sum_{n_1 + \dots + n_\ell = n} p_\equivalence([n_1,\dots,n_\ell]) \bigg(\sum_{i=1}^\ell \frac{n_i}{n} \log \bigg(\frac{n}{n_i}\bigg)\bigg)\bigg)n
\end{equation}
is asymptotically $|\tau|n$, which will of course entail that $\langle H_B \rangle \asymptotic |\tau|n$. Note that
\begin{equation}\label{eq:entropy_balls_bins}
    \sum_{i=1}^\ell \frac{n_i}{n} \log \bigg(\frac{n}{n_i}\bigg)
\end{equation}
is the Shannon entropy of the distribution on $\{1,\dots,\ell\}$ which assigns to each $1 \leq i \leq \ell$ the weight $\frac{n_i}{n}$. Thus we can use $\log(\ell) = |\tau|$ to bound the formula
\begin{equation}\label{eq:expedted_entropy}
    \sum_{n_1 + \dots + n_\ell = n} p_\equivalence([n_1,\dots,n_\ell]) \bigg(\sum_{i=1}^\ell \frac{n_i}{n} \log \bigg(\frac{n}{n_i}\bigg)\bigg)
\end{equation}
from above. Hence $|\tau|n$ is an upper bound on (\ref{eq:rewritten_expected_boltzmann}).

We will next bound (\ref{eq:expedted_entropy}) from below by using Proposition \ref{proposition:law_of_large_numbers}. For every $1 \leq i \leq \ell$ and $j \in \mathbb{Z}_+$ we let $X_j^i$ denote a random Bernoulli variable with success probability $2^{-|\tau|}$. Intuitively speaking, $X_j^i$ is an indicator function for the event ``the $j$th element received the $i$th $1$-type''. Now, for every $1 \leq i \leq \ell$ and for all $\delta > 0$ the law of large numbers implies that

\medskip

$\lim_{n \to \infty} \Pr\bigg[\bigg|\frac{1}{n}\sum_{j = 1}^n X_j^i - 2^{-|\tau|}\bigg| < \delta \bigg] = 1.$

\medskip

\noindent
Thus it follows from the union bound that the following probability
\begin{equation}\label{eq:prob_law_large_numbers}
\Pr\bigg[\forall \ 1 \leq i \leq \ell : \bigg|\frac{1}{n}\sum_{j = 1}^n X_j^i - 2^{-|\tau|}\bigg| < \delta \bigg]
\end{equation}
approaches $1$ as $n \to \infty$. Fix $\delta > 0$. For every $n$ we let $I_n^\delta$ denote the following set:
\[\bigg\{(n_1,\dots,n_\ell) \mid n_1 + \dots + n_\ell = n \text{ and } \forall \ 1\leq i \leq \ell: \bigg |\frac{n_i}{n} - 2^{-|\tau|} \bigg| < \delta \bigg\}.\]

The set $I_n^\delta$ includes the tuples $(n_1, \dots, n_\ell)$, where the numbers add up to $n$ and are very close to each other. The probability result above intuitively means that a randomly chosen tuple is almost always in $I_n^\delta$. Thus, roughly speaking, we only need to consider models, where the points are split between all of the types very evenly.

Let $n$ be large enough so that (\ref{eq:prob_law_large_numbers}) is larger than $(1 - \delta)$. For every $(n_1,\dots,n_\ell) \in I_n^\delta$ we want to estimate the formula (\ref{eq:entropy_balls_bins}) from below. Fix a tuple $(n_1,\dots,n_\ell) \in I_n^\delta$. Now, for every $1 \leq i \leq \ell$ we have $2^{-|\tau|} - \delta < n_i / n < 2^{-|\tau|} + \delta$, which also entails that
$n / n_i > 2^{|\tau|}/(1 + \delta2^{|\tau|})$.
%On the other hand there is $n$ large enough so that 
%\[2^{|\tau|} - \delta < \frac{n}{n_i} < 2^{|\tau|} + \delta.\]
%Let $n = \max\{n_1,n_2\}$.
Thus for every $1 \leq i \leq \ell$ we have that
%\[(2^{-|\tau|} - \delta) \log\bigg(\frac{2^{|\tau|}}{(1 + \delta2^{|\tau|})}\bigg) < \frac{n_i}{n} \log\bigg(\frac{n}{n_i}\bigg).\]
%\[(2^{-|\tau|} - \delta) \log(2^{|\tau|} - \delta) < \frac{n_i}{n} \log\bigg(\frac{n}{n_i}\bigg)\]
%Hence
\[2^{|\tau|}(2^{-|\tau|} - \delta) \log\bigg(\frac{2^{|\tau|}}{(1 + \delta2^{|\tau|})}\bigg) < \sum_{i=1}^\ell \frac{n_i}{n} \log \bigg(\frac{n}{n_i}\bigg).\]
Now we can bound the formula (\ref{eq:expedted_entropy}) from below by
\[2^{|\tau|}(2^{-|\tau|} - \delta) \log\bigg(\frac{2^{|\tau|}}{(1 + \delta2^{|\tau|})}\bigg) \cdot \sum_{(n_1,\dots,n_\ell) \in I_n^\delta} p_\equivalence([n_1,\dots,n_\ell]).\]
Notice that the right-hand side expresses the probability that a random $\tau$-model $\mathfrak{A}$ of size $n$ belongs to $[n_1,\dots,n_\ell]$, for some $(n_1,\dots,n_\ell) \in I_n^\delta$, which we know is at least $(1 - \delta)$, since we chose $n$ to be large enough. Thus we have, for every $\delta > 0$ and $n$ sufficiently large, the following lower bound for the formula (\ref{eq:expedted_entropy}):
\[f(\delta) := 2^{|\tau|}(2^{-|\tau|} - \delta) \log\bigg(\frac{2^{|\tau|}}{(1 + \delta2^{|\tau|})}\bigg) \cdot (1 - \delta).\]
Observe that $f(\delta) \to |\tau|$ as $\delta \to 0$. Hence, for every $\varepsilon > 0$ we have that $(1 - \varepsilon)|\tau| < f(\delta)$, for sufficiently small $\delta$. Combining this with our upper bound of $|\tau|n$ for (\ref{eq:rewritten_expected_boltzmann}) one can easily show that (\ref{eq:rewritten_expected_boltzmann}) is asymptotically $|\tau|n$. This concludes our proof of the following theorem.

\begin{theorem}\label{thm:expected_boltzmann}
    $\langle H_B \rangle \asymptotic |\tau|n$.
\end{theorem}

\subsection{Expected description complexity}

In this subsection we show that $\langle C \rangle \asymptotic n$. Let $M$ be an equivalence class of $\equivalence$. For a $1$-type $\pi$ we denote $|\pi|_M := |\{w \in W \mid (\MM, w) \vDash \pi\}|$, where $\MM \in M$. The number $|\pi|_M$ is the number of points that satisfy the type $\pi$ in the models of the class $M$. Since we will focus on a single class $M$ we will omit the subscript in the sequel. Let $\pi_m$ be the  $1$-type with the largest number of points in the models of the class $M$. Let $I := \{1 \leq i \leq 2^{|\tau|} \mid |\pi| \geq 1\}$. The set $I$ consists of the indices of types that are realized in the class $M$. In this subsection we show that the formula size required to define such a class $M$ is in the order of $\min(n, 2(n-|\pi_m|))$.

For upper bounds, we define a class $M$ via two different formulas, one of them using the largest type $\pi_m$ defined above:
\begin{align*}
    \varphi_1 &:= \bigwedge\limits_{i \in I} \di^{\geq |\pi_i|} \psi(\pi_i) \\
    \varphi_2 &:= %\neg \di^{\geq 1} \neg 
    \bo^{< 1}\big(\bigvee\limits_{i \in I} \psi(\pi_i)\big) \land \bigwedge\limits_{i \in I\setminus \{m\}} \di^{\geq |\pi_i|} \psi(\pi_i) \land \bigwedge\limits_{i \in I\setminus \{m\}} %\neg \di^{\geq |\pi_i|+1} 
    \bo^{< |\pi_i| + 1} \neg\psi(\pi_i)
\end{align*}
It is easy to verify that $\size(\varphi_1) = n + \mathcal{O}_{|\tau|}(1)$ and $\size(\varphi_2) = 2(n-|\pi_m|)+\mathcal{O}_{|\tau|}(1)$.

For the lower bounds, we utilize a formula size game $\fsc^\tau_r(\AA, \BB)$ for $\gmlu[\tau]$. The rules of the game are the same as in the $\mlu[\tau]$-game except the \dimove s and \bomove s are replaced with the following new moves:
\begin{itemize}
    \item \ddimove: S chooses a number $d \in \Nset$. If $r \leq d$, the game ends and D wins. Otherwise, for every $(\MM, w) \in \AA$, S chooses $d$ different points $v \in W$. Let $\AA'$ be the set of models $(\MM, v)$ chosen this way. For every $(\MM, w) \in \BB$, S chooses $n-d+1$ different points $v \in W$. Let $\BB'$ again be the set of models chosen. The next position of the game is $(r-d, \AA', \BB')$.
    \item \bbomove: The same as a \ddimove\ with the roles of $\AA$ and $\BB$ switched.
\end{itemize}
The equivalent of Theorem \ref{pelitoimii} can be proved for this new game in a very similar manner.

We now define the starting model sets of our formula size game. As before, we assume the domain of the models is $W = \{1, \dots, n\}$. Let $\AA_0 := \{(\MM, 1)\}$, where $\MM \in M$. We additionally assume that the points 1 and 2 of the model are propositionally equivalent. We do not need to fix the model $\MM$ any more precisely but note that there is only one model in the set $\AA_0$. Now let $(i, j) \in I \times I$ with $i \neq j$ and let $w \in W$ be the largest number with $(\MM, w) \vDash \pi_i$. The model $\MM_{i \rightarrow j}$ has $(\MM_{i \rightarrow j}, w) \vDash \pi_j$ and is otherwise identical to $\MM$. In other words, $\MM_{i \rightarrow j}$ has one less point of the type $\pi_i$ and one more of the type $\pi_j$ compared to $\MM$. We let $\BB_0 := \{(\MM_{i\rightarrow j}, 1) \mid i, j \in I, i \neq j\}$. There are $|I|^2$ models in the set $\BB_0$. Note that all models in $\AA_0$ and $\BB_0$ have propositionally equivalent starting points. 
%As in the proof of Lemma \ref{umldwins}, one of the two sets of models in the game will always contain only versions of the model $\MM$ with different points. We denote this set by $\AA$ and the opposite one with models $\MM_{i \rightarrow j}$ by $\BB$ regardless of which is on the left and which on the right in any given position.

Let us now consider the formula size game $\fsc^\tau_{r_0}(\AA_0, \BB_0)$. For any position $(r, \AA, \BB)$ %or $(r, \BB, \AA)$ 
of this game, we define a directed graph $\G(\AA, \BB) := (V, E)$ by setting $V := I$ and $(i, j) \in E$ iff there are propositionally equivalent $(\MM, w) \in \AA$ and $(\MM_{i \rightarrow j}, v) \in \BB$. We call a set $C \subseteq \{i^+, i^- \mid i \in I\}$ a \textbf{cover} of $\G(\AA, \BB)$ if for every $(i, j) \in E$ we have $i^+ \in C$ or $j^- \in C$. The cost of a cover $C$ is 

\medskip

$
r(C) := \sum\limits_{i^+ \in C} |\pi_i|_M + \sum\limits_{i^- \in C} |\pi_i|_M.
$

\medskip

We are now ready for the crucial Lemma of this subsection.

\begin{lemma}\label{gmlulemma}
Let $P := (r, \AA, \BB)$ %(or $(r, \BB, \AA)$)
be a position of the game $\fsc^\tau_{r_0}(\AA_0, \BB_0)$ and let $R(P) := \min\{r(C) \mid C \text{ is a cover of } \G(\AA, \BB)\}$. If $r < R(P)$, then D has a winning strategy in the game from the position $P$.
\end{lemma}
\begin{proof}
    We show that any move S makes either leads to D winning the game immediately or maintains the conditions of the claim given the correct choice by D.

    \textbf{\propmove}: Since $R(P) > 0$, there are propositionally equivalent pointed models on both sides of the game so clearly D wins if S makes any \propmove. 
        
        %
        
        %\textbf{\negmove}: The condition of the claim is symmetrical with regards to the sides of the sets $\AA$ and $\BB$ so it is trivially maintained.

        \textbf{\lormove}: Let $\AA_1, \AA_2 \subseteq \AA$ and $r_1, r_2 \geq 1$ be the choices of S and let $P_1 = (r_1, \AA_1, \BB)$ and $P_2 = (r_2, \AA_2, \BB)$. For each edge $e = (i, j) \in E$, there are propositionally equivalent models $(\MM, w) \in \AA$ and $(\MM_{i \rightarrow j}, v) \in \BB$. Since $\AA_1 \cup \AA_2 = \AA$, every model $(\MM, w) \in \AA$ is in $\AA_1$ or $\AA_2$ so every edge of the graph $\G(\AA, \BB)$ is present in at least one of the graphs $\G(\AA_1, \BB)$ and $\G(\AA_2, \BB)$. We claim that $r_1 < R(P_1)$ or $r_2 < R(P_2)$. Assume for contradiction that $r_1 \geq R(P_1)$ and $r_2 \geq R(P_2)$. Then there is a cover $C_1$ of $\G(\AA_1, \BB)$ with $r(C_1) \leq r_1$ and the same for $P_2$. Now $C_1 \cup C_2$ is a cover of $\G(\AA, \BB)$. Additionally $r(C_1 \cup C_2) \leq r(C_1) + r(C_2) \leq r_1 + r_2 \leq r$. This means that $R(P) \leq r$, which is a contradiction with the condition $r < R(P)$. Thus D can choose a position that maintains the condition of the claim.

        \textbf{\landmove}: Very similar to the above case with the models in $\BB$ split between $\BB_1$ and $\BB_2$.

        \textbf{\ddimove}: %First let $P = (r, \AA, \BB)$. 
        Let $d \in \Nset$ be the number chosen by S.
        For each $(\MM, w) \in \AA$, S chooses $d$ different points from the model $\MM$. Let $A$ be the set of all points chosen this way.
        For each $(\MM_{i \rightarrow j}, w) \in \BB$, let $B_{i \rightarrow j} $ be the set of $n-d+1$ points chosen by S. Let $\mathrm{tp}_{\MM}(X)$ be the set of types realized by a set $X$ of points in the model $\MM$. We consider the following two cases:
        \begin{enumerate}
            \item The model $\MM$ has at least $d+1$ points with types from $\mathrm{tp}_{\MM}(A)$. Let $e = (i, j) \in E$. The model $\MM_{i \rightarrow j}$ only differs from $\MM$ by the type of one point so $\MM_{i \rightarrow j}$ has at least $d$ points that realize types from $\mathrm{tp}_{\MM}(A)$. Since $|B_{i \rightarrow j}| = n-d+1$, there is at least one point in $B_{i \rightarrow j}$ with a type from $\mathrm{tp}_{\MM}(A)$. Thus there are propositionally equivalent $(\MM, w') \in \AA'$ and $(\MM_{i \rightarrow j}, v') \in \BB'$. Thus the edge $e = (i,j)$ is still present in the graph $\G(\AA', \BB')$ of the following position. This applies for every $e \in E$ so $R(P') = R(P)$.

            %$|A| = d$. In other words, for all versions $(\MM, w) \in \AA$ of the model $\MM$, the same $d$ points were chosen. If there is a point $v \in W \setminus A$ with a type from $\mathrm{tp}_{\MM}(A)$, then there are again $d+1$ points in $\MM$ with types from $\mathrm{tp}_{\MM}(A)$ and the above argument holds. Thus we assume below that the points in $A$ are exactly all the points with types from $\mathrm{tp}_{\MM}(A)$.
            
            % 
            
            \item The model $\MM$ has exactly $d$ points with types from $\mathrm{tp}_{\MM}(A)$. Now $A$ is the set of those $d$ points. We first consider edges $e = (i,j) \in E$ with $\pi_i \notin \mathrm{tp}_{\MM}(A)$ or $\pi_j \in \mathrm{tp}_{\MM}(A)$. For any edge of this kind, the model $\MM_{i \rightarrow j}$ has at least $d$ points with types from $\mathrm{tp}_{\MM}(A)$ so at least one of the $n-d+1$ points in $B_{i \rightarrow j}$ has a type from $\mathrm{tp}_{\MM}(A)$. As in case 1, this means that all these edges are still present in the graph $\G(\AA', \BB')$.

            Let us then consider the rest of the edges $e = (i,j) \in E$ with $\pi_i \in \mathrm{tp}_{\MM}(A)$ and $\pi_j \notin \mathrm{tp}_{\MM}(A)$. For an edge of this kind, the model $\MM_{i \rightarrow j}$ has only $d-1$ points with types from $\mathrm{tp}_{\MM}(A)$. Thus if S chooses the $n-d+1$ points of $B_{i \rightarrow j}$ to be exactly the points with types not in $\mathrm{tp}_{\MM}(A)$, then $\MM_{i \rightarrow j}$ has no propositionally equivalent counterpart on the other side and the edge $e$ is not present in the graph $\G(\AA', \BB')$.

            We then consider the condition of the claim in the position $P' = (r-d, \AA', \BB')$. By the above arguments, the only way S could remove edges when moving from $\G(\AA, \BB)$ to $\G(\AA', \BB')$, was to choose in each version of the model $\MM$ exactly all of the points that satisfy some set $\mathrm{tp}_{\MM}(A)$ of types. Any edge eliminated this way originates from an index $i$ of a type in $\mathrm{tp}_{\MM}(A)$. All of these edges can be covered via the cover $C_A = \{i^+ \mid \pi_i \in \mathrm{tp}_{\MM}(A)\}$. The cost of this cover is the total number of points of the model $\MM$ with types from $\mathrm{tp}_{\MM}(A)$. Since $A$ contains exactly all points with types from $\mathrm{tp}_{\MM}(A)$, we have $r(C_A) = |A| = d$. Let $C'$ be a cover of $\G(\AA', \BB')$ with minimal cost so $R(P') = r(C')$. Now $C' \cup C_A$ is a cover of $\G(\AA, \BB)$ with cost $R(P') + d$. Thus $r < R(P) \leq R(P') + d$ so $r-d < R(P')$ and the condition of the claim is maintained. 
        \end{enumerate}
        
        %Finally we cover the case of a \ddimove\ from a position $P = (r, \BB, \AA)$. 
        \textbf{\bbomove}: Similar to the \ddimove\ with $n-d+1$ points chosen from models in $\AA$ and $d$ points chosen from models in $\BB$. Full details in the Appendix.
\end{proof}

By the above Lemma, the formula size required to define a class $M$ of the equivalence $\equivalence$ comes down to calculating the minimum cost of a cover.

\begin{theorem}\label{thm:uml_formula_size_lower_bound}
Let $M$ be an equivalence class of the relation $\equivalence$ and let $\pi$ be the propositional type with most satisfying points in models in $M$. If the formula $\varphi \in \gmlu[\tau]$ defines the class $M$, then $\varphi$ has size at least $\min(n, 2(n-|\pi|))$.
\end{theorem}
\begin{proof}
Let $s = \min(n, 2(n-|\pi|))$. We use the above Lemma to show that D has a winning strategy in the game $\fsc^\tau_s(\AA_0, \BB_0)$, thus proving the claim.

It suffices to show that the minimum cost of a cover of $\G(\AA_0, \BB_0) = (V, E)$ is equal to $s$. First we see that $\G(\AA_0, \BB_0)$ is a complete irreflexive directed graph. We begin by noting that $C^+ := \{i^+ \mid i \in I\}$ is a cover with cost $n$ and adding any $i^-$ or replacing $i^+$ with $i^-$ does not reduce the cost. Thus if all indices are used, $C^+$ is a minimum cost cover. Next, we consider covers $C_i$, where there is an index $i \in I$ with $\{i^+, i^-\} \cap C_i = \emptyset$. Note that $i$ is the only such index. Indeed, if there were a second such index $j$, then the edge $(i, j)$ would not be covered. Now, for any $j \in I$, $j \neq i$ we have $j^+ \in C_i$ since it is the only way to cover the edge $(j, i)$. In the same way $j^- \in C_i$ since the edge $(i, j)$ must be covered. Thus $C_i = \{j^+, j^- \mid j \in I, j \neq i\}$. The cost of $C_i$ is
\[
r(C_i) = \sum\limits_{j^+ \in C_i} |\pi_j| + \sum\limits_{j^- \in C_i} |\pi_j| = n - |\pi_i| + n - |\pi_i| = 2(n - |\pi_i|).
\]
The cost minimal cover of this type is clearly the one where $i$ is the index of the type with the most satisfying points. Thus the minimal cover size is $\min(n, 2(n-|\pi|))$.
\end{proof}

\begin{theorem}\label{thm:expected_formula_size}
    $\langle C \rangle \asymptotic n$.
\end{theorem}
\begin{proof}
    Since $C(M) \leq n + \mathcal{O}_{|\tau|}(1)$ for any equivalence class $M$, we have $\langle C \rangle \leq n + \mathcal{O}_{|\tau|}(1)$. For the lower bound, recall from the previous section that for any $\delta > 0$ and $n$ sufficiently large we have that
    \[\sum_{(n_1 , \dots , n_\ell) \in I_n^\delta}p_\equivalence([n_1,\dots,n_\ell]) > (1 - \delta).\]
    Observe that if $(n_1, \dots, n_\ell) \in I_n^\delta$, for $\delta$ sufficiently small, then Theorem \ref{thm:uml_formula_size_lower_bound} entails that 
    $C([n_1,\dots,n_\ell]) \geq n$
    as every $1$-type is realized less than $n/2$-times. Thus, for any $\delta > 0$ and $n$ sufficiently large, we have $\langle C \rangle \geq (1 - \delta)n$. Using these bounds it is easy to show  $\langle C \rangle \asymptotic n$.
\end{proof}

The desired relation between Boltzmann entropy and description complexity now follows directly from Theorems \ref{thm:expected_boltzmann} and \ref{thm:expected_formula_size}.

\begin{corollary}\label{thm:entropy_and_complexity}
    $\langle H_B \rangle \asymptotic |\tau|\langle C \rangle$
\end{corollary}

\begin{comment}
\begin{remark}

As for $\mlu[\tau]$, we again comment on free use of negation instead of negation normal form. We can again add a move for negation to the game. The results for $\gmlu[\tau]$ are unaffected, as negations only concern the part of formula size depending on $|\tau|$, which we treat as a constant with respect to $n$. The same goes for removing $\bo^{< d}$ as a primitive operator.

\end{remark}
\end{comment}

\section{FO: Expected description complexity for polyadic\\ vocabularies}

We saw in the previous section that the ratio of expected Boltzmann entropy of an isomorphism class and its $\gmlu$-description complexity is asymptotically the size of the underlying fixed vocabulary. Given that the main characteristic of $\gmlu$ is that it can characterize finite monadic structures up to isomorphism, one might guess that a similar behaviour would extend to $\mathrm{FO}$, which can characterize arbitrary finite structures up to isomorphism. The purpose of this section is to show that surprisingly this is not the case: the expected description complexity grows faster than the expected Boltzmann entropy.

Given a relation symbol $R$ we will use $ar(R)$ to denote its arity. Fix a finite relational vocabulary $\tau$ and let $m := \max \{ar(R) \mid R \in \tau\}$. For the rest of this section we will assume that $m \geq 2$. The following result, which fails for unary vocabularies, is established in \cite{Fagin77}.

\begin{proposition}
    The number of non-isomorphic $\tau$-models of size $n$ is asymptotically $2^{p(n)}/n!$,
    where $p(n) = \sum_{R \in \tau} n^{ar(R)}$.
\end{proposition}

In \cite{PikhurkoV09} the authors mention (without a proof) that with high probability, defining a single graph of size $n$ up to isomorphism in $\mathrm{FO}$ requires a sentence of size $\Omega\big(\frac{n^2}{\log(n)}\big)$. Here we prove a version of this statement for an arbitrary (but finite) relational vocabulary. For the proof, recall that $\equiv_{\mathrm{FO}[\tau]}$ is over $\mathrm{Mod}_n(\tau)$.

\begin{theorem}\label{thm:average_case_lower_bound_fo}
    With high probability we have that $C_{\mathrm{FO}[\tau]}(M) = \Omega\big(\frac{n^m}{\log(n)}\big)$, when the isomorphism class $M$ is selected uniformly at random.
\end{theorem}
\begin{proof}
    The proof is a counting argument: we will show that the ratio between ``short'' formulas and isomorphism classes of models of size $n$ approaches $0$ as $n$ increases. Fix $s \geq 2$. We will start by bounding the number of $\mathrm{FO}[\tau]$-sentences of size $s$ in which only variables from the set $\{x_1,\dots,x_n\}$ occur. Note that the number of atomic $\tau$-formulas over $\{x_1,\dots,x_n\}$ is $\sum_{R \in \tau} n^{ar(R)} =: N_\tau$. Each $\mathrm{FO}[\tau]$-sentence of size $s$ can be viewed as a labeled tree with $s$ nodes, the labels being literals and symbols from the set $\{\land,\vee,\exists,\forall\}$. Since a tree with $s$ nodes has $s - 1$ edges, each $\mathrm{FO}[\tau]$-sentence of size $s$ can be encoded using, say, $10(s - 1)\log(N_\tau + 4)$ bits. Thus there are at most $2^{10(s-1)\log(N_r + 4)}$ $\mathrm{FO}[\tau]$-sentences of size $s$.

    Using this bound we can also easily bound the number of $\mathrm{FO}[\tau]$-sentences of size \emph{at most} $s$ (and at least two). Indeed, the number of such sentences is at most $\sum_{i = 2}^s 2^{10(i-1)\log(N_\tau + 4)}$, which is bounded from above by $2^{10s \log(N_\tau + 4)}$.

    Now the number of non-isomorphic $\tau$-models of size $n$ is asymptotically 

    \medskip
    
    $2^{p(n)}/n! \geq 2^{n^m}/n! \geq 2^{n^m - n\log(n)} = 2^{\big(1 - \frac{\log(n)}{n^{m-1}}\big)n^m}$
    
    \medskip
    
    \noindent
    Note that $\log(N_\tau + 4) \leq d \log(n)$, for some $d > 0$ and $n$ sufficiently large. Thus, if we set $s = cn^m/\log(n)$, for some $c > 0$ that will be specified later, then we have that
    
    \smallskip 
    
    $2^{\big(10\frac{cn^m}{\log(n)}\big)\log(N_\tau + 4)} \leq 2^{\big(10\frac{cn^m}{\log(n)}\big) d \log(n)} = 2^{10cdn^m}$
    
    \smallskip
    
    \noindent
    for sufficiently large $n$. Combining these two estimates we have that
    \[\frac{2^{\big(10\frac{cn^m}{\log(n)}\big)\log(N_\tau + 4)}}{2^{p(n)}/n!} \leq \bigg(2^{10cd + \frac{\log(n)}{n^{m-1}} - 1} \bigg)^{n^m}\]
    Since $\log(n)/n^{m-1} \to 0$, by taking $c$ sufficiently small and $n$ sufficiently large we have that $2^{10cd + \frac{\log(n)}{n^{m-1}} - 1} < 1$. Thus with high probability we have that $C_{\mathrm{FO}[\tau]}(M) = \Omega\big(\frac{n^m}{\log(n)}\big)$.
\end{proof}

\begin{remark}
    Since for every isomorphism class $M$ we have that $C_{\mathrm{FO}[\tau]}(M) = \mathcal{O}(n^m)$, there is a small gap between this upper bound and the lower bound established in Theorem \ref{thm:average_case_lower_bound_fo}. Even in the case of graphs it seems an open problem to determine the average case $\mathrm{FO}$-description complexity of an isomorphism class, see \cite{PikhurkoV09} for more discussion.
\end{remark}

Consider now the partition $\equivalence_{\mathrm{FO}[\tau]}$ of $\mathrm{Mod}_n(\tau)$. In Appendix \ref{appendix:boltzmann_entropy_vs_description_complexity} we use Theorem \ref{thm:average_case_lower_bound_fo} to establish the following result.

\begin{proposition}\label{prop:boltzmann_entropy_versus_description_complexity}
    Expected description complexity of $\equivalence_{\mathrm{FO}[\tau]}$ grows asymptotically faster than its expected Boltzmann entropy.
\end{proposition}

Note that Proposition \ref{prop:boltzmann_entropy_versus_description_complexity} does not follow immediately from Theorem \ref{thm:average_case_lower_bound_fo}, since there we consider the uniform distribution over the isomorphism classes, while here we need to consider $p_{\equivalence_{\mathrm{FO}[\tau]}}$ which a priori could place negligible probabilities on isomorphism classes with high description complexity. However, it follows from the results of \cite{Fagin77} that for large $n$ the distribution $p_{\equivalence_{\mathrm{FO}[\tau]}}$ is quite close to the uniform distribution.

\medskip

\medskip

\medskip

\noindent
\textbf{Acknowledgments.}\ \ \ 
Antti Kuusisto and Miikka Vilander were supported by the Academy of Finland project \emph{Explaining AI via Logic} (XAILOG), grant number 345612 (Kuusisto). Antti Kuusisto was also supported by the Academy of Finland project \emph{Theory of computational logics}, grant numbers 324435, 328987 (to December 2021); 352419, 352420 (January to August 2022), 352419, 353027 (from September 2022).

\bibliographystyle{plainurl}
\bibliography{arxivstacs}

\section{Appendix}

\subsection{Proof of Proposition \ref{prop:shannon_boltzmann}}\label{appendix:shannon_boltzmann}

Letting $\{M_i \mid i \in I\}$ enumerate the equivalence classes of $\equivalence$, we have the following chain of identities.
    \begin{align*}
        &H_S(\equivalence) + \langle H_B \rangle \\
        = &- \sum\limits_{i \in I} p_\equivalence (M_i) \log p_\equivalence(M_i) + \sum\limits_{i \in I} p_\equivalence (M_i) \log(|[M_i]_\equivalence|) \\
        = &- \sum\limits_{i \in I} p_\equivalence (M_i) \log (|[M_i]_\equivalence|/|\mathcal{M}|) + \sum\limits_{i \in I} p_\equivalence (M_i) \log(|[M_i]_\equivalence|) \\
        = &- \sum\limits_{i \in I} p_\equivalence (M_i) (\log (|[M_i]_\equivalence|) - \log(|\mathcal{M}|)) + \sum\limits_{i \in I} p_\equivalence (M_i) \log(|[M_i]_\equivalence|) \\
        = &- \sum\limits_{i \in I} p_\equivalence (M_i) \log (|[M_i]_\equivalence|) + \sum\limits_{i \in I}p_\equivalence (M_i)\log(|\mathcal{M}|) + \sum\limits_{i \in I} p_\equivalence (M_i) \log(|[M_i]_\equivalence|) \\
        = &\log(|\mathcal{M}|)\sum\limits_{i \in I}p_\equivalence (M_i) \\
        = &\log(|\mathcal{M}|)\qedhere
    \end{align*}

\subsection{Proof of Proposition \ref{prop:largest_equivalence_class}}

The following standard calculation shows that if $n$ is large enough, then the probability that a random $\tau$-model of size $n$ does not realize all the $1$-types is less than $1/2$.
\begin{equation*}
\begin{split}
    &\Pr[\exists \pi : \MM \text{ does not realize } \pi]  \leq \sum_\pi \Pr[\MM \text{ does not realize } \pi] \\
         & = \sum_\pi \prod_{a\in W} \Pr[a \text{ does not realize } \pi]
          = \sum_\pi \prod_{a\in W} (1 - \Pr[a \text{ does realize } \pi]) \\
         & = \sum_\pi \prod_{a\in W} (1 - 2^{-k}) 
          = n(1 - 2^{-k})^n \to 0, \text{ as } n \to \infty
    \end{split}
    \end{equation*}
In the inequality we used union bound while in the first equality we used the fact that the events ``$a$ does not realize $\pi$'', for $a\in W$, are independent.

\subsection{Proof of Lemma \ref{lem:umldwins} continued}

\textbf{\lormove}: We show that for any \lormove\ S makes, D can choose one of the following positions $P_1, P_2$ that satisfies both conditions (a) and (b). %We assume $\AA$ is on the left since a \lormove\ with $\BB$ on the left is almost identical to a \landmove\ with $\BB$ on the right, which we handle below.
    
    Let $\AA_1, \AA_2$ and $r_1, r_2$ be the choices of S. We assume $\AA_1, \AA_2 \neq \emptyset$. Let $\pi_i$ be a type. If $\pi_i$ is of kind 1, then $\pi_i$ is still of kind 1 in both following positions and $h_i(P) = 2k = h_i(P_1) = h_i(P_2)$, since $\BB$ remains unchanged in both positions. If $\pi_i$ is of kind 2, then there are propositionally equivalent $(\MM_0, j) \in \AA$ and $(\MM_i, l) \in \BB$. We have $(\MM_0, j) \in \AA_1$ or $(\MM_0, j) \in \AA_2$ so $\pi_i$ is still a type of kind 2 in one of the following positions and $h_i(P_1) = 2k$ or $h_i(P_2) = 2k$. Similarly if $\pi_i$ is of kind 3, then $(\MM_0, i) \in \AA_1$ or $(\MM_0, i) \in \AA_2$ so $\pi_i$ remains a type of kind 3 in one of the following positions and $h_i(P_1) = h_i(P)$ or $h_i(P_2) = h_i(P)$. 
    %We have seen that 
    %\[
    %\sum\limits_{1\leq i\leq n}h_i(P_1) + \sum\limits_{1\leq i\leq n}h_i(P_2) \geq %\sum\limits_{1\leq i\leq n}h_i(P).
    %\]
    Furthermore, each type with positive hardness in $P$ still has positive hardness in at least one of $P_1$ or $P_2$ so $\poshard(P_1) + \poshard(P_2) \geq \poshard(P)$. Thus
    \begin{align*}
    h(P_1) + h(P_2) &= \sum\limits_{1\leq i\leq n}h_i(P_1) +  \sum\limits_{1\leq i\leq n}h_i(P_2) + \poshard(P_1) + \poshard(P_2) - 2 \\
    &\geq \sum\limits_{1\leq i\leq n}h_i(P) + \poshard(P) - 1 - 1 = h(P) - 1.
    \end{align*}
    Now $r_1 + r_2 = r - 1 < h(P) - 1 \leq h(P_1) + h(P_2)$ so we have $r_1 < h(P_1)$ or $r_2 < h(P_2)$.
    In addition, since all types are of the same kind as in position $P$, condition (b) still holds.

\subsection{Proof of Lemma \ref{gmlulemma} continued}

\textbf{\bbomove}: Let $d \in \Nset$ be the number chosen by S. For each $(\MM, w) \in \AA$, S chooses $n-d+1$ different points from the model $\MM$. Let $A$ be the set of all points chosen this way. For each $(\MM_{i \rightarrow j}, w) \in \BB$, let $B_{i \rightarrow j} $ be the set of $d$ points chosen by S. Let $\mathrm{tp}_{\MM}(X)$ be the set of types realized by  the set $X$ of points in the model $\MM$. We consider the following two cases:
        \begin{enumerate}
            \item The model $\MM$ has at least $n-d+2$ points with types from $\mathrm{tp}_{\MM}(A)$. Let $e = (i,j) \in E$. The model $\MM_{i \rightarrow j}$ only differs from $\MM$ by the type of one point so $\MM_{i \rightarrow j}$ has at least $n-d+1$ points that satisfy types from $\mathrm{tp}_{\MM}(A)$. Since $|B_{i \rightarrow j}| = d$, there is at least one point in $B_{i \rightarrow j}$ with a type from $\mathrm{tp}_{\MM}(A)$. Thus there are propositionally equivalent $(\MM, w') \in \AA'$ and $(\MM_{i \rightarrow j}, v') \in \BB'$. Thus the edge $e = (i,j)$ is still present in the graph $\G(\AA', \BB')$ of the following position. This applies for every $e \in E$ so $R(P') = R(P)$.

           %$|A| = n - d + 1$. In other words, for all versions $(\MM, w) \in \AA$ of the model $\MM$, the same $n-d+1$ points were chosen. If there is a point $v \in W \setminus A$ with a type from $\mathrm{tp}_{\MM}(A)$, then there are again $n-d+2$ points in $\MM$ with types from $\mathrm{tp}_{\MM}(A)$ and the above argument holds. Thus we assume below that the points in $A$ are exactly all the points with types from $\mathrm{tp}_{\MM}(A)$.
            
            % 
            
             \item The model $\MM$ has exactly $n-d+1$ points with types from $\mathrm{tp}_{\MM}(A)$. Now $A$ is the set of those $n-d+1$ points. We first consider edges $e = (i,j) \in E$ with $\pi_i \notin \mathrm{tp}_{\MM}(A)$ or $\pi_j \in \mathrm{tp}_{\MM}(A)$. For any edge of this kind, the model $\MM_{i \rightarrow j}$ has at least $n-d+1$ points with types from $\mathrm{tp}_{\MM}(A)$ so at least one of the $d$ points in $B_{i \rightarrow j}$ has a type from $\mathrm{tp}_{\MM}(A)$. As in case 1, this means that all these edges are still present in the graph $\G(\AA', \BB')$.

            Let us then consider the rest of the edges $e = (i,j) \in E$ with $\pi_i \in \mathrm{tp}_{\MM}(A)$ and $\pi_j \notin \mathrm{tp}_{\MM}(A)$. For an edge of this kind, the model $\MM_{i \rightarrow j}$ has only $n-d$ points with types from $\mathrm{tp}_{\MM}(A)$. Thus if S chooses the $d$ points of $B_{i \rightarrow j}$ to be exactly the points with types not in $\mathrm{tp}_{\MM}(A)$, then $\MM_{i \rightarrow j}$ has no propositionally equivalent counterpart on the other side and the edge $e$ is not present in the graph $\G(\AA', \BB')$.

            We then consider the condition of the claim in the position $P' = (r-d, \BB', \AA')$. We saw above that S can only eliminate an edge $e = (i,j)$ if $\pi_i \in \mathrm{tp}_{\MM}(A)$, $\pi_j \notin \mathrm{tp}_{\MM}(A)$ and $\mathrm{tp}_{\MM_{i \rightarrow j}}(B_{i \rightarrow j}) \subseteq \mathrm{tp}_{\MM}(W) \setminus \mathrm{tp}_{\MM}(A)$. Thus we denote $\mathrm{tp}(B) := \mathrm{tp}_{\MM}(W) \setminus \mathrm{tp}_{\MM}(A)$. All edges of this kind can be covered via the cover $C_B = \{j^- \mid \pi_j \in \mathrm{tp}(B)\}$. The cost of this cover is the total number of points with types from $\mathrm{tp}(B)$ in the model $\MM$. By the definition of $\mathrm{tp}(B)$ the cost is $r(C_B) = n - |A| = n-(n-d+1) = d-1$. Let $C'$ be a cover of $\G(\AA', \BB')$ with minimal cost so $R(P') = r(C')$. Now $C' \cup C_B$ is a cover of $\G(\AA, \BB)$ with cost $R(P') + d-1$. Thus $r < R(P) \leq R(P') + d-1$ so $r-d < R(P')$ and the condition of the claim is maintained.  
        \end{enumerate}

%\subsection{Proof of Theorem \ref{thm:entropy_and_complexity}}\label{appendix:proof_of_entropy_complexity}

%Theorems \ref{thm:expected_boltzmann} and \ref{thm:expected_formula_size} yield that $\langle H_B \rangle \asymptotic |\tau|n$ and $\langle C \rangle \asymptotic n$ respectively. Combining these we have that
%\[\lim_{n \to \infty} \frac{\langle H_B \rangle}{|\tau|\langle C \rangle} = \lim_{n \to \infty} \frac{\langle H_B \rangle}{|\tau|n|} \times \frac{|\tau|n}{n} \times \frac{n}{|\tau|\langle C \rangle} = 1,\]
%which implies that $\langle H_B \rangle \asymptotic |\tau| \langle C \rangle$.

\subsection{Proof of Proposition \ref{prop:boltzmann_entropy_versus_description_complexity}}\label{appendix:boltzmann_entropy_vs_description_complexity}

In this section we use $\equivalence$ to denote $\equivalence_{\mathrm{FO}[\tau]}$. Our goal is to show that the expected Boltzmann entropy of $\equivalence$ grows asymptotically slower than its expected description complexity.

We start by bounding the expected Boltzmann entropy from above. For every equivalence class $M$ of $\equivalence$ we have by Proposition \ref{prop:stirling_approximation} that \[\log(|M|) \leq \log(n!) = n\log(n) - n\log(e) + \Theta(\log(n)),\]
which in turn implies that $\langle H_B \rangle \leq n\log(n) - n\log(e) + \Theta(\log(n))$. 

Next we will derive a lower bound on the expected description complexity of $\equivalence$. Let $c$ be a constant such that with high probability $C_{\mathrm{FO}[\tau]}(M) \geq c\big(\frac{n^m}{\log(n)}\big)$. (Theorem \ref{thm:average_case_lower_bound_fo} guarantees that such a constant exists.) In \cite{Fagin77} it was proved that with high probability a random $\tau$-model is rigid, i.e., it has no non-trivial automorphism. Since the isomorphism class of a rigid $\tau$-model is of size $n!$, we have that with high probability a random member of $\equivalence$ has size $n!$. Using a union bound argument we have that
\begin{equation}\label{eq:probability_result}
    \lim_{n\to\infty} \Pr\bigg[C_{\mathrm{FO}[\tau]}(M) \geq c\bigg(\frac{n^m}{\log(n)}\bigg) \text{ and } |M| = n!\bigg] = 1
\end{equation}
In particular, the above probability is at least, say, $1/2$ when $n$ is large enough. In other words, for $n$ large enough, at least half of the isomorphism classes (of models of size $n$) have size $n!$ and their description complexity is at least $c\big(\frac{n^m}{\log(n)}\big)$.

Now we can bound the expected description complexity from below. First, we have that
\begin{align*}
    &\sum_M p_\equivalence(M) C_{\mathrm{FO}[\tau]}(M) \geq \sum_{C_{\mathrm{FO}[\tau]}(M) \geq c\big(\frac{n^m}{\log(n)}\big)} p_\equivalence(M) C_{\mathrm{FO}[\tau]}(M) \\
    &\geq c\bigg(\frac{n^m}{\log(n)}\bigg) \sum_{C_{\mathrm{FO}[\tau]}(M) \geq c\big(\frac{n^m}{\log(n)}\big)} p_\equivalence(M) 
    = c\bigg(\frac{n^m}{\log(n)}\bigg) \frac{1}{2^{p(n)}} \sum_{C_{\mathrm{FO}[\tau]}(M) \geq c\big(\frac{n^m}{\log(n)}\big)} |M|.
\end{align*}
We want a constant lower bound on the expression
\[\frac{1}{2^{p(n)}} \sum_{C_{\mathrm{FO}[\tau]}(M) \geq c\big(\frac{n^m}{\log(n)}\big)} |M|\]
which expresses the probability that the isomorphism class of random model of size $n$ has description complexity at least $c\big(\frac{n^m}{\log(n)}\big)$. Using Equation (\ref{eq:probability_result}), we have for $n$ large enough the following estimates:
\begin{align*}
    &\frac{1}{2^{p(n)}} \sum_{C_{\mathrm{FO}[\tau]}(M) \geq c\big(\frac{n^m}{\log(n)}\big)} |M| \geq \frac{1}{2^{p(n)}} \sum_{\substack{C_{\mathrm{FO}[\tau]}(M) \geq c\big(\frac{n^m}{\log(n)}\big) \\ |M| = n!}} |M|
    = \frac{n!}{2^{p(n)}} \sum_{\substack{C_{\mathrm{FO}[\tau]}(M) \geq c\big(\frac{n^m}{\log(n)}\big) \\ |M| = n!}} 1 \\
    &\geq \frac{n!}{2^{p(n)}} (1/2) \frac{2^{p(n)}}{n!} 
    = 1/2.
\end{align*}
Thus for $n$ large enough we have that $\langle C_{\mathrm{FO}[\tau]} \rangle \geq (1/2) c\big(\frac{n^m}{\log(n)}\big)$, which certainly grows faster than $n\log(n) - n\log(e) + \Theta(\log(n)) \geq \langle H_B \rangle$.

\end{document}